\documentclass[a4paper,11pt]{amsart}
\usepackage{amsmath}
\usepackage{amssymb}
\usepackage{mathrsfs}
\usepackage{enumerate}
\usepackage{ifthen}
\usepackage{graphicx}
\usepackage{subfig}
\usepackage{caption}
\usepackage[T1]{fontenc} 
\usepackage{tikz}
\usepackage{cite}

\nonstopmode \numberwithin{equation}{section}
\newtheorem{theorem}{Theorem}[section]

\newtheorem{conj}{Conjecture}[section]

\newtheorem{lemma}{Lemma}[section]
\newtheorem{problem}{Problem}[section]

\theoremstyle{remark}
\theoremstyle{definition}
\newtheorem{remark}{Remark}[section]
\newtheorem{definition}{Definition}[section]

\theoremstyle{plain}

\numberwithin{equation}{section}
\numberwithin{theorem}{section}

\newcounter{minutes}
\divide\time by 60
\newcounter{hours}
\multiply\time by 60
\addtocounter{minutes}{-\time}

\begin{document}

\title{Sharp Estimates of Logarithmic Coefficients for a Certain Class of Starlike Functions}

\author{Molla Basir Ahamed}
\address{Molla Basir Ahamed, Department of Mathematics, Jadavpur University, Kolkata-700032, West Bengal, India.}
\email{mbahamed.math@jadavpuruniversity.in}

\author{Sanju Mandal}
\address{Sanju Mandal, Department of Mathematics, Jadavpur University, Kolkata-700032, West Bengal, India.}
\email{sanjum.math.rs@jadavpuruniversity.in, sanju.math.rs@gmail.com}


\def\thefootnote{}
\footnotetext{ {\tiny File:~\jobname.tex,
printed: \number\year-\number\month-\number\day,
          \thehours.\ifnum\theminutes<10{0}\fi\theminutes }
} \makeatletter\def\thefootnote{\@arabic\c@footnote}\makeatother

\begin{abstract}
 In this article, we investigate the extremal properties of logarithmic coefficients for the class $\mathcal{S}_{ch}^*$ of starlike functions associated with the hyperbolic cosine function. We establish the sharp upper bounds for the initial logarithmic coefficients $\gamma_n$ for $n=1, 2, 3$, and determine the precise bound for the second Hankel determinant $H_{2,1}(F_f/2)$ within this class. Furthermore, we extend our analysis to the inverse functions, deriving sharp estimates for the logarithmic inverse coefficients and the corresponding second Hankel determinant $|H_{2,1}(F_{f^{-1}}/2)|$. Additionally, we provide sharp bounds for the moduli differences of both logarithmic and inverse logarithmic coefficients. The sharpness of all obtained inequalities is verified through the construction of specific extremal functions.
\end{abstract}

\thanks{}
\maketitle
\pagestyle{myheadings}
\markboth{M. B. Ahamed and S. Mandal}{Sharp Logarithmic Coefficient Estimates for Starlike Functions}

\section{\bf Introduction}
Let $\mathcal{H}$ denote the class of holomorphic functions $f$ in the open unit disk $\mathbb{D}=\{z\in\mathbb{C}: |z|<1\}$. Then $\mathcal{H}$ is a locally convex topological vector space endowed with the topology of uniform convergence over compact subsets of $\mathbb{D}$. Let $\mathcal{A}$ denote the class of functions $f\in\mathcal{H}$ such that $f(0)=0$ and $f^{\prime}(0)=1$ \textit{i.e.}, the function $f$ is of the form
\begin{align}\label{eq-1.1}
	f(z)=z+ \sum_{n=2}^{\infty}a_nz^n,\; \mbox{for}\; z\in\mathbb{D}.
\end{align} 
Let $\mathcal{S}$ denote the subclass of all functions in $\mathcal{A}$ which are univalent. For a comprehensive understanding of the theory of univalent functions and their significance in coefficient problems, we refer to the books \cite{Duren-1983-NY,Goodman-1983}.

\begin{definition}\label{def-1.1}
Let $f$ and $g$ be two analytic functions in the unit disk $\mathbb{D}$. Then $f$ is said to be subordinate to $g$, written as $f\prec g$ or $f(z)\prec g(z)$, if there exists a function $\omega$, analytic in $\mathbb{D}$ with $w(0)=0$, $|w(z)|<1$ such that $f(z)=g(w(z))$ for $z\in\mathbb{D}$. Moreover, if $g$ is univalent in $\mathbb{D}$ and $f(0)=g(0)$, then $f(\mathbb{D})\subseteq g(\mathbb{D})$.
\end{definition}

In $1994$, Ma and Minda \cite{Ma-Minda-1994} gave a unified presentation of various subclasses of starlike and convex functions by replacing the subordinate function $(1+z)/(1-z)$ by a more general analytic function $\varphi$ with positive real part and normalized by the conditions $\varphi(0) =1$, $\varphi^{\prime}(0)> 0$ and $\varphi$ maps $\mathbb{D}$ onto univalently a starlike region with respect to $1$ and symmetric with respect to the real axis. Ma and Minda \cite{Ma-Minda-1994} have introduced the following general class that envelopes several well-known classes as special cases
\begin{align*}
	\mathcal{S}^*[\varphi]=\bigg\{f\in\mathcal{A}: \frac{zf^{\prime}(z)}{f(z)}\prec\varphi(z)\bigg\}.
\end{align*} In the literature, functions belonging to the class $\mathcal{S}^*[\varphi]$ are known as the Ma-Minda starlike functions. For $-1\leq B<A\leq 1$, the class $\mathcal{S}^*[(1+Az)/(1+Bz)]:= \mathcal{S}^*[A,B]$
is called the class of Janowski starlike functions, introduced by Janowski \cite{Janowski-APM-1973}. The class $\mathcal{S}^*[\beta]$ of starlike functions of order $\beta $, where $ 0\leq\beta<0 $, is defined by taking $\varphi(z)= (1+(1-2\beta)z)/(1-z)$. Note that $\mathcal{S}^*=\mathcal{S}^*[0]$ is the classical class of starlike functions. By taking 
\begin{align*}
	\varphi(z)= 1+ \frac{2}{\pi^2}\left(\log\left((1+\sqrt{z})/(1-\sqrt{z})\right)\right)^2,
\end{align*} we obtain the class $\mathcal{S}^*[\varphi]=\mathcal{S}_p$ of parabolic starlike functions, introduced by Rønning \cite{Rønning-PAMS-1993}. The coefficient problem and many other geometric properties for functions in the class $\mathcal{S}^*[\varphi]$ have been studied extensively in (see \cite{Deniz-BMMSS-2021, Goodman-APM-1991,Kazimoglu-Deniz-Srivastava-COAT-2024,Lecko-Sim-RM-2019,Keogh-Merkes-PAMS-1969}). For example, Deniz \cite{Deniz-BMMSS-2021} has studied the sharp coefficient problem for the function
\begin{align*}
	\varphi(z)=e^{z+\frac{\lambda}{2}z^2} \;\;(z\in\mathbb{C}, \lambda\geq 1), 
\end{align*}
which is a generated function of generalized telephone numbers. In $2015$, Mendiratta \textit{et al.} \cite{Mendiratta-Nagpal-Ravichandran-BMMSS-2015} obtained the structural formula, inclusion relations, coefficient estimates, growth and distortion results, subordination theorems and various radii constants for the exponential function $\varphi(z)= e^z$. \vspace{1.2mm}

In $2019$, Cho \textit{et al.} \cite{Cho et.al.-BIMS-2019} studied the class $\mathcal{S}^{*}_{S}:=\mathcal{S}^{*}(1+\sin z)$ and the class $\mathcal{S}^{*}_{\cos}:=\mathcal{S}^{*}(\cos z)$ has been studied by \cite{Marımuthu-Uma-Bulboaca-HJM-2023} and \cite{Bano-Raza-BIMS-2021}. In the recent years, the class of starlike functions associated with cosine hyperbolic functions has been discussed by Riaz \textit{et al.} \cite{Riaz-Lecko-Raza-IJS-2023} which is defined as follows:
\begin{definition}
Let $f\in\mathcal{A}$. Then $f\in\mathcal{S}^*_{ch}$ if, and only if, 
\begin{align*}
	\frac{zf^{\prime} (z)}{f(z)} \prec \varphi_{0}(z) \;\;\;\;z\in\mathbb{D},
\end{align*}
where
\begin{align*}
	\mathbb{C}\ni z\mapsto\varphi_{0}(z):= z+\cosh (z)=z+ \frac{1}{2}\left(\exp(z) +\exp(-z)\right).
\end{align*}
\begin{figure}[!htb]
	\begin{center}
		\subfloat[Unit disk $z$-plane]{
			\includegraphics[width=0.45\linewidth]{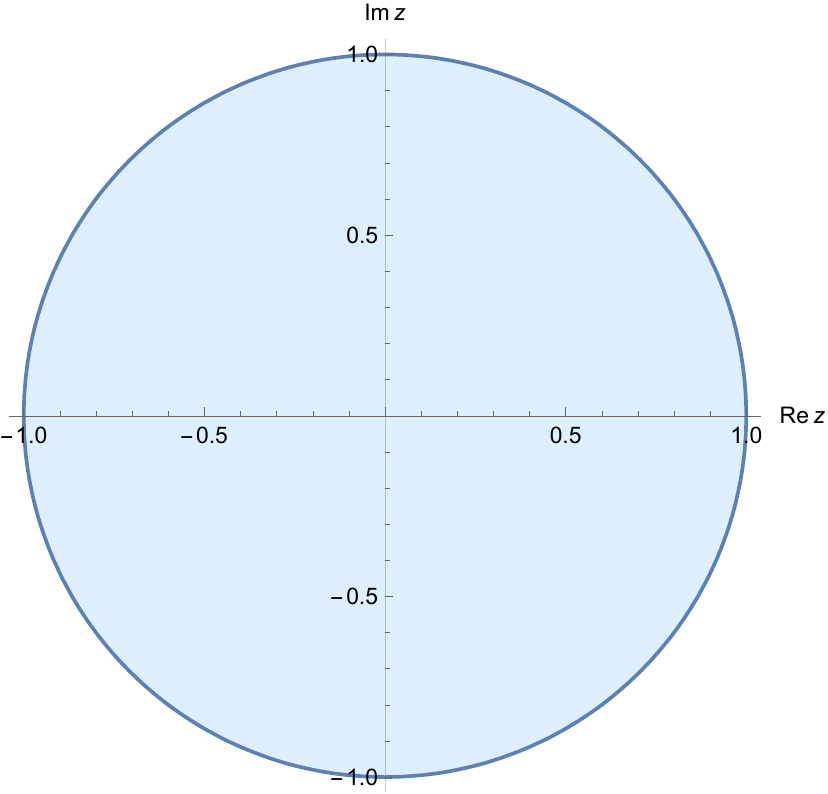}}
		\hspace{2mm}
		\subfloat[$w$-plane]{
			\includegraphics[width=0.50\linewidth]{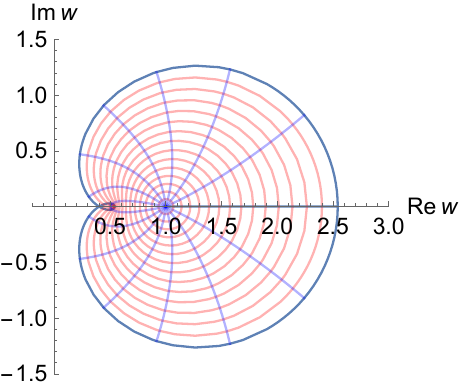}}
	\end{center}
	\caption{The geometrical representation of $z+\cosh (z)$.}
\end{figure}\vspace{2mm}

It is important to note that the function $\varphi$ need not be univalent in $\mathbb{D}$ (see \cite[Observation 1.2]{Riaz-Lecko-Raza-IJS-2023}). A function $f\in\mathcal{S}^*_{ch}$ if, and only if, there exists an analytic function $\varphi\in\mathcal{H}$, satisfying $\varphi\prec\varphi_{0}$ such that
\begin{align*}
	f(z)=z\exp\left(\int_{0}^{z}\frac{\varphi(t)-1}{t}dt\right), \;\;\;\;z\in\mathbb{D}.
\end{align*}
\end{definition}
\subsection{New problem formulations for the class $\mathcal{S}^*_{ch}$} Establishing the upper bounds for coefficients has remained a central theme in geometric function theory, as these estimates reveal fundamental structural properties of analytic functions. A primary challenge in this pursuit is the identification of suitable extremal functions that demonstrate the sharpness of such bounds. We observe that while coefficient problems involving the Hankel determinant for the class $\mathcal{S}^*_{ch}$ have been extensively explored, the corresponding determinants composed of logarithmic coefficients for this class have not yet garnered significant research attention. Furthermore, the sharp bounds for individual logarithmic coefficients for the class $\mathcal{S}^*_{ch}$ remain largely unaddressed. This gap in the literature serves as the primary motivation for the present study. \vspace{1.2mm}

The primary motivation of this study is to provide a comprehensive characterization of the logarithmic functionals for functions subordinated to the hyperbolic cosine function. To this end, we formulate and resolve the following key problems.
\begin{problem}\label{Prob-1}
	Let $f \in \mathcal{S}_{ch}^*$ be a starlike function associated with the hyperbolic cosine function. Determine the sharp upper bounds for the modulus of the logarithmic coefficients $\gamma_n$ for $n\in\mathbb{N}$.
\end{problem}
\begin{problem}\label{Prob-2}
	For $f\in \mathcal{S}_{ch}^*$, find the sharp bound for the second Hankel determinant $H_{2,1}(F_f/2) = \gamma_1 \gamma_3 - \gamma_2^2$ consisting of logarithmic coefficients.
\end{problem}
\begin{problem}\label{Prob-3}
	Let $f^{-1}$ be the inverse of a function $f$ in $\mathcal{S}_{ch}^*$. Determine the sharp bounds of the modulus of the inverse logarithmic coefficients $\Gamma_n$.
\end{problem}
\begin{problem}\label{Prob-4}
	Establish the sharp bound for the second Hankel determinant $|H_{2,1}(F_{f^{-1}}/2)|$ for the inverse functions belonging to the class $\mathcal{S}_{ch}^*$.
\end{problem}
\begin{problem}\label{Prob-5}
	Analyze the difference between the moduli of consecutive logarithmic coefficients $|\gamma_{n+1}| - |\gamma_n|$ and their inverse counterparts. Determine if a uniform sharp bound exists for these differences for functions in the class $\mathcal{S}_{ch}^*$.
\end{problem}
In this article, our aim is give affirmative answers to the Problems \ref{Prob-1} to \ref{Prob-5}. These problems include finding the sharp bound of the Hankel determinant of logarithmic coefficients, the sharpness of $\gamma_1, \gamma_2,$ and $\gamma_3$. In the subsequent sections, we will discuss our findings and provide a background study on these topics. The organization of this paper is as follows: In Section 2, we establish the sharp bounds for $\gamma_1, \gamma_2,$ and $\gamma_3$ and $|H_{2,1}(F_f/2)|$ for functions $f\in\mathcal{S}^*_{ch}$. In Section 3, we establish the sharp bound of  $\Gamma_1$, $\Gamma_2$ and $|H_{2,1}(F_{f^{-1}}/2)|$ for function in the class $\mathcal{S}^*_{ch}$. In Section 4, we establish results finding sharp bounds of the moduli differences of logarithmic coefficients and inverse counterpart for functions in the class $\mathcal{S}^*_{ch}$. The proofs of the results are discussed in detail in each respective section.

\section{\bf Sharp bound of logarithmic coefficients for functions in the class $\mathcal{S}^*_{ch}$:}
For $f\in\mathcal{S}$, we define the logarithmic coefficients $\gamma_{n}(f)$ by
\begin{align}\label{eq-2.1}
	F_{f}(z):=\log\dfrac{f(z)}{z}=2\sum_{n=1}^{\infty}\gamma_{n}(f)z^n, \;\; z\in\mathbb{D},\;\;\log 1:=0.
\end{align}
The logarithmic coefficients $\gamma_{n}$ for functions in the class $\mathcal{S}$ play a vital role in Milin’s conjecture (\cite{Milin-1977-ET}, see also \cite[p.155]{Duren-1983-NY}).
Milin conjectured that for $f\in\mathcal{S}$ and $n\geq2$,
\begin{align*}
	\sum_{m=1}^{n}\sum_{k=1}^{m}\left(k|\gamma_{k}|^2 -\frac{1}{k}\right)\leq 0,
\end{align*}
where the equality holds if, and only if, $f$ is a rotation of the Koebe function. De Branges \cite{Branges-AM-1985} has proved Milin conjecture which confirmed the famous Bieberbach conjecture. On the other hand, one of reasons for more attention has been given to the logarithmic coefficients is that the sharp bound for the class $\mathcal{S}$ is known only for $\gamma_{1}$ and $\gamma_{2}$, namely
\begin{align*}
	|\gamma_{1}|\leq 1, \;\; |\gamma_{2}|\leq \frac{1}{2}\left(1+2e^{-2}\right) =0.635\ldots
\end{align*}
It is still an open problem to find the sharp upper bounds for absolute value of $\gamma_{n}$, $n\geq 3$, for functions in the class $\mathcal{S}$. Estimating the modulus of logarithmic coefficients for functions $f\in\mathcal{S}$ and various sub-classes has been considered recently by several authors. For more information on logarithmic coefficients, we refer to \cite{Ali-Allu-PAMS-2018, Roth-PAMS-2007, Ali-Allu-Thomas-CRMCS-2018, Cho-Kowalczyk-kwon-Lecko-Sim-RACSAM-2020,Girela-AASF-2000,Thomas-PAMS-2016,Ahamed-Mandal-F-2026,Ahamed-Mandal-UMJ-2026,Mandal-Ahamed-LMJ-2024,Mandal-Ahamed-LMJ-2024,Mandal-Ahamed-Zaprawa-MS-2025} and references therein.\vspace{2mm}

The evaluation of Hankel determinants has been a major concern in geometric function theory, where these determinants are formed by employing coefficients of analytic functions $f$ that are represented by \eqref{eq-1.1} in the unit disk $\mathbb{D}$. Hankel matrices (and determinants) have emerged as fundamental elements in different areas of mathematics, finding a wide range of applications (see \cite{Ye-Lim-FCM-2016}). The primary objective of this study is to determine the sharp bound of logarithmic coefficients and the Hankel determinants involving the logarithmic coefficients. To begin, we present the definitions of Hankel determinants in situations where $f\in \mathcal{A}$.\vspace{1.2mm}

The Hankel determinant $H_{q,n}(f)$ of Taylor's coefficients of functions $f\in\mathcal{A}$ represented by \eqref{eq-1.1} is defined for $q,n\in\mathbb{N}$ as follows:
\begin{align*}
H_{q,n}(f):=\begin{vmatrix}
a_{n} & a_{n+1} &\cdots& a_{n+q-1}\\ a_{n+1} & a_{n+2} &\cdots& a_{n+q} \\ \vdots & \vdots & \vdots & \vdots \\ a_{n+q-1} & a_{n+q} &\cdots& a_{n+2(q-1)}
\end{vmatrix}.
\end{align*}
The extensive exploration of the sharp bounds of the Hankel determinants for starlike, convex, and other function classes have been undertaken in various studies (see \cite{Raza-Riza-Thomas-BAMS-2023, Kowalczyk-Lecko-BAMS-2022, Mandal-Ahamed-LMJ-2024}), and their precise bounds have been successfully established. \vspace{1.2mm}

Differentiating \eqref{eq-2.1} and then using \eqref{eq-1.1}, a simple computation shows that 
\begin{align}\label{eq-2.2}
	\begin{cases}
		\gamma_{1}=\dfrac{1}{2}a_{2},\vspace{2mm}\\ \gamma_{2}=\dfrac{1}{2} \left(a_{3} -\dfrac{1}{2}a^2_{2}\right), \vspace{2mm}\\ \gamma_{3} =\dfrac{1}{2}\left(a_{4}- a_{2}a_{3} +\dfrac{1}{3}a^3_{2}\right) \vspace{2mm}.
	\end{cases}
\end{align}
In $2022$, Kowalczyk and Lecko \cite{Kowalczyk-Lecko-BAMS-2022} proposed a Hankel determinant $H_{q,n}(F_f/2)$ whose elements are the logarithmic coefficients of $f\in\mathcal{S}$, realizing the extensive use of these coefficients. It follows that
\begin{align}\label{eq-2.3}
		H_{2,1}(F_f/2):=\gamma_1\gamma_3-\gamma_2^2=\frac{1}{48}\left(a_2^4-12a_3^2+12a_2a_4\right).
\end{align}
Furthermore, $H_{2,1}(F_{f}/2)$ is invariant under rotation, since for $f_{\theta}(z):=e^{-i\theta}f(e^{i\theta}z)$, $\theta\in\mathbb{R}$ when $f\in\mathcal{S}$, we have
\begin{align*}
	H_{2,1}(F_{f_{\theta}}/2)=\frac{e^{4i\theta}}{48}\left(a^4_2 - 12 a^2_3 + 12 a_2 a_4\right)=e^{4i\theta}H_{2,1}(F_{f}/2).
\end{align*}
Let $\mathcal{P}$ be the class of all analytic functions $p$ in the unit disk $\mathbb{D}$ satisfying $p(0)=1$ and $\mbox{Re}\;p(z)>0$ for $z\in\mathbb{D}$. Therefore, every $p\in\mathcal{P}$ can be represented as
\begin{align}\label{eq-2.4}
	p(z)=1+\sum_{n=1}^{\infty}c_n z^n,\; z\in\mathbb{D}.
\end{align}
Elements of the class $\mathcal{P}$ are called  Carath$\acute{e}$odory functions. It is known that $|c_n|\leq 2$, $n\geq 1$ for a functions $p\in\mathcal{P}$ (see \cite{Duren-1983-NY}). The Carath$\acute{e}$odory class $\mathcal{P}$ and it's coefficients bound play a significant role in establishing the bound of Hankel determinants.\vspace{2mm}

Now, we state some lemmas, which will be useful to establish our main results. Parametric representations of the coefficients are often useful in finding the bound for Hankel determinants, and in this regard, Libera and Zlotkiewicz (see \cite{Libera-Zlotkiewicz-PAMS-1982, Libera-Zlotkiewicz-PAMS-1983}) obtained the parameterizations of possible values of $c_2$ and $c_3$, which are Taylor coefficients for functions with positive real part.

\begin{lemma}\cite{Libera-Zlotkiewicz-PAMS-1982,Libera-Zlotkiewicz-PAMS-1983}\label{lem-2.1}
If $p\in\mathcal{P}$ is of the form \eqref{eq-2.4} with $c_1\geq 0$, then 
\begin{align}\label{eq-2.5}
	&c_1=2\tau_1,\\\label{eq-2.6} &c_2=2\tau^2_1 +2(1-\tau^2_1)\tau_2
\end{align}
and
\begin{align}\label{eq-2.7}
	c_3= 2\tau^3_1  + 4(1 -\tau^2_1)\tau_1\tau_2 - 2(1 - \tau^2_1)\tau_1\tau^2_2 + 2(1 - \tau^2_1)(1 - |\tau_2|^2)\tau_3
\end{align}
for some $\tau_1\in[0,1]$ and $\tau_2,\tau_3\in\overline{\mathbb{D}}:= \{z\in\mathbb{C}:|z|\leq 1\}$.\vspace{1.2mm}
	
For $\tau_1\in\mathbb{T}:=\{z\in\mathbb{C}:|z|=1\}$, there is a unique function $p\in\mathcal{P}$ with $c_1$ as in \eqref{eq-2.5}, namely
\begin{align*}
	p(z)=\frac{1+\tau_1 z}{1-\tau_1 z}, \;\;z\in\mathbb{D}.
\end{align*}
	
For $\tau_1\in\mathbb{D}$ and $\tau_2\in\mathbb{T}$, there is a unique function $p\in\mathcal{P}$ with $c_1$ and $c_2$ as in \eqref{eq-2.5} and \eqref{eq-2.6}, namely
\begin{align*}
	p(z)=\frac{1+(\overline{\tau_1}\tau_2 +\tau_1)z+\tau_2 z^2}{1 +(\overline{\tau_1}\tau_2 -\tau_1)z-\tau_2 z^2}, \;\;z\in\mathbb{D}.
\end{align*}
	
For $\tau_1,\tau_2\in\mathbb{D}$ and $\tau_3\in\mathbb{T}$, there is a unique function $p\in\mathcal{P}$ with $c_1,c_2$ and $c_3$ as in \eqref{eq-2.5}--\eqref{eq-2.7}, namely
\begin{align*}
	p(z)=\frac{1+(\overline{\tau_2}\tau_3+\overline{\tau_1}\tau_2 +\tau_1)z +(\overline{\tau_1}\tau_3+ \tau_1\overline{\tau_2}\tau_3 +\tau_2)z^2 +\tau_3 z^3}{1 +(\overline{\tau_2}\tau_3+ \overline{\tau_1}\tau_2 -\tau_1)z +(\overline{\tau_1}\tau_3- \tau_1\overline{\tau_2}\tau_3 -\tau_2)z^2 -\tau_3 z^3}, \;\;z\in\mathbb{D}.
\end{align*}
\end{lemma}

\begin{lemma}\cite{Cho-Kim-Sugawa-JMSJ-2007}\label{lem-2.2}
Let $A, B, C$ be real numbers and
\begin{align*}
	Y(A,B,C):=\max\{|A+ Bz +Cz^2| +1-|z|^2: z\in\overline{\mathbb{D}}\}.
\end{align*}
\noindent{(i)} If $AC\geq 0$, then
\begin{align*}
	Y(A,B,C)=\begin{cases}
		|A|+|B|+|C|, \;\;\;\;\;\;\;\;\;\;\;\;\;|B|\geq 2(1-|C|), \vspace{2mm}\\ 1+|A|+\dfrac{B^2}{4(1-|C|)}, \;\;\;\;\;|B|< 2(1-|C|).
	\end{cases}
\end{align*}
\noindent{(ii)} If $AC<0$, then
\begin{align*}
	Y(A,B,C)=\begin{cases}
		1-|A|+\dfrac{B^2}{4(1-|C|)}, \;\;\;\;-4AC(C^{-2}-1)\leq B^2\land|B|< 2(1-|C|),\vspace{2mm} \\ 1+|A|+\dfrac{B^2}{4(1+|C|)}, \;\;\;\; B^2<\min\{4(1+|C|)^2,-4AC(C^{-2}-1)\}, \vspace{2mm} \\ R(A,B,C), \;\;\;\;\;\;\;\;\;\;\;\;\;\;\;\;\;\;\; otherwise,
	\end{cases}
\end{align*}
where
\begin{align*}
	R(A,B,C):= \begin{cases}
		|A|+|B|-|C|, \;\;\;\;\;\;\;\;\;\;\;\;\; |C|(|B|+4|A|)\leq |AB|, \vspace{2mm}\\ -|A|+|B|+|C|, \;\;\;\;\;\;\;\;\;\;\; |AB|\leq |C|(|B|-4|A|), \vspace{2mm}\\ (|C| +|A|)\sqrt{1-\dfrac{B^2}{4AC}}, \;\;\; otherwise.
	\end{cases}
\end{align*}
\end{lemma}

\begin{lemma}\cite{Ma-Minda-1994}\label{lem-2.3}
Let $p\in\mathcal{P}$ be given by \eqref{eq-2.1}. Then
\begin{align*}
	|c_2 -vc^2_1|\leq\begin{cases}
		-4v +2 \;\;\;\; v<0,\\ 2 \;\;\;\;\;\;\;\;\;\;\;\;\;\; 0\leq v\leq 1, \\ 4v -2 \;\;\;\;\;\;\; v>1.
	\end{cases}
\end{align*}
For $v<0$ or $v>1$, the equality holds if, and only if,
\begin{align*}
	h(z)=\frac{1+z}{1-z}
\end{align*}
or one of its rotations. If $0 < v < 1$, then the equality is true if, and only if,
\begin{align*}
	h(z)=\frac{1+z^2}{1-z^2}
\end{align*}
or one of its rotations.
\end{lemma}

\begin{lemma}\cite{Ali-BMMSS-2001}\label{lem-2.4}
Let $p\in\mathcal{P}$ be given by \eqref{eq-2.1} with $0\leq B\leq 1$ and $B(2B-1)\leq D\leq B$. Then 
\begin{align*}
	|c_3 -2Bc_1 c_2 +Dc^3_1|\leq 2.
\end{align*}
\end{lemma}

The significance of logarithmic coefficients in geometric function theory has led to a growing interest in finding sharp bound of logarithmic coefficients and the Hankel determinants with these coefficients. We obtain the following sharp bound of initial logarithmic coefficients $\gamma_{1}$, $\gamma_{2}$, $\gamma_{3}$ and $H_{2,1}(F_f/2)$ for functions in the class $\mathcal{S}^*_{ch}$.
\begin{theorem}\label{Th-22.11}
Let $f(z)=z+a_2z^2+a_3z^3+\cdots\in\mathcal{S}^*_{ch}$ and $ \gamma_{1}, \gamma_{2} $, $ \gamma_{3} $ are given by \eqref{eq-2.2}. Then we have 
\begin{align*}
	|\gamma_n|\leq \frac{1}{2n}\;\;\;\;\mbox{for}\;\;n=1,2,3.
\end{align*}
The inequality is sharp for the following functions:
\begin{align}\label{SF-2.8}
	f_1(z)=z\exp\left(\int_{0}^{z}\frac{t +\cosh(t)-1}{t}dt\right)=z+ z^2 +\frac{3}{4}z^3 +\frac{5}{12}z^4 +\cdots,\;\;z\in\mathbb{D}.
\end{align}

\begin{align}\label{SF-2.9}
	f_2(z)=z\exp\left(\int_{0}^{z}\frac{t^2 +\cosh(t^2)-1}{t}dt\right) =z+ \frac{1}{2}z^3 +\frac{1}{4}z^5 +\cdots,\;\;z\in\mathbb{D}.
\end{align}

\begin{align}\label{SF-2.10}
	f_3(z)=z\exp\left(\int_{0}^{z}\frac{t^3 +\cosh(t^3)-1}{t}dt\right) =z+ \frac{1}{3}z^4 +\frac{5}{36}z^7 +\cdots,\;\;z\in\mathbb{D}.
\end{align}
\end{theorem}
\begin{proof}[\bf Proof of Theorem \ref{Th-22.11}]
Let $f\in\mathcal{S}^*_{ch}$. Then there exists an analytic function $\omega\in\mathcal{H}$ with $\omega(0)=0$ and $|\omega(z)|<1$ for $z\in\mathbb{D}$, such that 
\begin{align}\label{eq-2.8}
	\frac{zf^{\prime}(z)}{f(z)}=\varphi_{0}(\omega(z))=\omega(z)+\cosh(\omega(z)),\;\; z\in\mathbb{D}.
\end{align}
Let $p\in\mathcal{P}$. Then, by using the definition of subordination, we have
\begin{align*}
	p(z)=\frac{1+\omega(z)}{1-\omega(z)}=1+c_1z+c_2z^2+\cdots,\; z\in\mathbb{D}.
\end{align*} 
Hence, it is evident that
\begin{align}\label{eq-2.9}
	\omega(z)&=\nonumber\frac{p(z)-1}{p(z)+1}\\&\nonumber=\frac{c_1}{2}z+\frac{1}{2}\left(c_2-\frac{c_1^2}{2}\right)z^2+\frac{1}{2}\left(c_3-c_1c_2+\frac{c_1^3}{4}\right)z^3\\&\quad+\frac{1}{2}\left(c_4-c_1c_3+\frac{3c_1^2c_2}{4}-\frac{c_2^2}{2}-\frac{c_1^4}{8}\right)z^4+\cdots
\end{align}
in $\mathbb{D}$. Then $p$ is analytic in $\mathbb{D}$ with $p(0)=1$ and has positive real part in $\mathbb{D}$. In view of \eqref{eq-2.8} together with $\varphi_{0}(\omega(z))$, a tedious computation shows that 
\begin{align}\label{eq-2.10}
	\nonumber\omega(z)+\cosh(\omega(z))&=1+\frac{1}{2}c_1z+\left(-\frac{1}{8}c_1^2 +\frac{1}{2}c_2\right)z^2+ \left(-\frac{1}{4}c_1c_2 +\frac{1}{2} c_3\right)z^3\\&\quad+\left(-\frac{1}{4}c_1c_3 +\frac{13}{384}c_1^4-\frac{1}{8}c_2^2+\frac{1}{2}c_4\right)z^4+\cdots
\end{align}
and 
\begin{align}\label{eq-2.11}
	\frac{zf^{\prime}(z)}{f(z)}&=1+a_2z+(2a_3-a_2^2)z^2+(3a_4-3a_2a_3+a_2^3)z^3\\&\quad+(4a_5-2a_3^2-4a_2a_4+4a_2^2a_3-a_2^4)z^4+\cdots\nonumber.
\end{align}
Thus, using \eqref{eq-2.10} and \eqref{eq-2.11}, we compute from \eqref{eq-2.8} that
\begin{align}\label{eq-2.12}
	\begin{cases}
		a_2=\dfrac{1}{2}c_1,\vspace{2mm}\\
		a_3=\dfrac{1}{16}c_1^2+\dfrac{1}{4}c_2,\vspace{2mm}\\
		a_4=-\dfrac{1}{96}c_1^3+\dfrac{1}{24}c_1c_2+\dfrac{1}{6}c_3.
	\end{cases}
\end{align}
\noindent{\bf Sharp bounds of $\gamma_{1}$:} By using \eqref{eq-2.2} and \eqref{eq-2.12}, it is easy to that
\begin{align}\label{Eq-2.13}
	|\gamma_{1}|=\;\vline\frac{1}{2}a_{2}\;\vline=\frac{1}{4}|c_{1}|\leq \frac{1}{2}.
\end{align}
The desired inequality is thus obtained. The function $f_1$, which is defined in \eqref{SF-2.8} gives the sharpness of the inequality \eqref{Eq-2.13}. \vspace{2mm}

\noindent{\bf Sharp bounds of $\gamma_{2}$:} From \eqref{eq-2.2} and \eqref{eq-2.12}, we see that	
\begin{align*}
	|\gamma_{2}|&=\;\vline\frac{1}{2} \left(a_{3} -\dfrac{1}{2}a^2_{2} \right)\vline \vspace{2mm} \\&= \;\vline\frac{1}{2}\left(\frac{1}{16}c_1^2 +\frac{1}{4}c_2-\frac{1}{2}\left(\frac{1}{2}c_1\right)^2\right)\vline \vspace{2mm} \\&= \frac{1}{8}\;\vline\; c_2-\frac{1}{4}c^2_{1} \;\vline. 
\end{align*}
By using Lemma \ref{lem-2.3}, we obtain the desired inequality
\begin{align}\label{Eq-2.17}
	|\gamma_{2}|\leq\frac{1}{4}.
\end{align}
The function $f_2$, which is defined in \eqref{SF-2.9} gives the sharpness of the inequality \eqref{Eq-2.17}.\vspace{2mm}

\noindent{\bf Sharp bounds of $\gamma_{3}$:} By using \eqref{eq-2.2} and \eqref{eq-2.12}, we obtain
\begin{align*}
	|\gamma_{3}| &=\;\vline\frac{1}{2}\left(a_{4}- a_{2}a_{3} +\frac{1}{3}a^3_{2}\right)\vline \vspace{2mm} \\&=\; \vline\frac{1}{2} \left(-\frac{1}{96}c_1^3+\frac{1}{24} c_1c_2+\frac{1}{6}c_3- \frac{1}{32} c^3_{1} -\frac{1}{8}c_1 c_2 +\frac{1}{24} c^3_{1}\right)\vline \vspace{2mm} \\&=\frac{1}{12}\;\vline\; c_3 -2Bc_1 c_2 +Dc^3_1\;\vline,
\end{align*}
where $B=\frac{1}{4}$ and $D=0$. Therefore, it is easy to see that $0\leq B\leq 1$ and the inequality $B(2B-1)\leq D\leq B$. By Lemma \ref{lem-2.4}, we have the inequality
\begin{align}\label{Eq-2.18}
	|\gamma_{3}|\leq\frac{1}{6}.
\end{align}
The function $f_3$, defined in \eqref{SF-2.10}, establishes the sharpness of the inequality \eqref{Eq-2.18}. This completes the proof.
\end{proof}\vspace{2mm}	

Next, we obtain a result finding the sharp bound of the second Hankel determinant $ H_{2,1}(F_f/2) $ with logarithmic coefficients for functions in the class $ \mathcal{S}^*_{ch} $.
\begin{theorem}\label{Th-2.1}
Let $f\in\mathcal{S}^*_{ch}$ and has the series representation $f(z)=z+a_2z^2+a_3z^3+\cdots$, and $ \gamma_{1}, \gamma_{2} $, $ \gamma_{3} $ are given by \eqref{eq-2.2}. Then we have 
\begin{align*}
	|H_{2,1}(F_f/2)|\leq \frac{1}{16}.
\end{align*}
The inequality is sharp for the function $f_2$, which is defined in \eqref{SF-2.9}.
\end{theorem}
\begin{proof}[\bf Proof of Theorem \ref{Th-2.1}]
Let $f\in\mathcal{S}^*_{ch}$. Then there exists an analytic function $\omega\in\mathcal{H}$ with $\omega(0)=0$ and $|\omega(z)|<1$ for $z\in\mathbb{D}$, such that 
\begin{align}\label{eq-2.13}
	\frac{zf^{\prime}(z)}{f(z)}=\varphi_{0}(\omega(z))=\omega(z)+\cosh(\omega(z)),\;\; z\in\mathbb{D}.
\end{align}
Let $p\in\mathcal{P}$. Then, in view of the definition of subordination, we have
\begin{align*}
	p(z)=\frac{1+\omega(z)}{1-\omega(z)}=1+c_1z+c_2z^2+\cdots,\; z\in\mathbb{D}.
\end{align*} 
It is easy to see that
\begin{align}\label{eq-2.14}
	\omega(z)&=\nonumber\frac{p(z)-1}{p(z)+1}\\&\nonumber=\frac{c_1}{2}z+\frac{1}{2}\left(c_2-\frac{c_1^2}{2}\right)z^2+\frac{1}{2}\left(c_3-c_1c_2+\frac{c_1^3}{4}\right)z^3\\&\quad+\frac{1}{2}\left(c_4-c_1c_3+\frac{3c_1^2c_2}{4}-\frac{c_2^2}{2}-\frac{c_1^4}{8}\right)z^4+\cdots
\end{align}
in $\mathbb{D}$. Then $p$ is analytic in $\mathbb{D}$ with $p(0)=1$ and has positive real part in $\mathbb{D}$. In view of \eqref{eq-2.13} together with $\varphi_{0}(\omega(z))$, a tedious computation shows that 
\begin{align}\label{eq-2.15}
	\omega(z)+\cosh(\omega(z))&=1+\frac{1}{2}c_1z+\left(-\frac{1}{8}c_1^2 +\frac{1}{2}c_2\right)z^2+\left(-\frac{1}{4}c_1c_2+\frac{1}{2} c_3\right)z^3\\&\quad+\left(-\frac{1}{4}c_1c_3+\frac{13}{384}c_1^4-\frac{1}{8}c_2^2+\frac{1}{2}c_4\right)z^4+\cdots\nonumber
\end{align}
and 
\begin{align}\label{eq-2.16}
	\frac{zf^{\prime}(z)}{f(z)}&=1+a_2z+(2a_3-a_2^2)z^2+(3a_4-3a_2a_3+a_2^3)z^3\\&\quad+(4a_5-2a_3^2-4a_2a_4+4a_2^2a_3-a_2^4)z^4+\cdots\nonumber.
\end{align}
Thus, using \eqref{eq-2.15} and \eqref{eq-2.16}, we see from \eqref{eq-2.13} that
\begin{align}\label{eq-2.17}
	\begin{cases}
		a_2=\dfrac{1}{2}c_1,\vspace{2mm}\\
		a_3=\dfrac{1}{16}c_1^2+\dfrac{1}{4}c_2,\vspace{2mm}\\
		a_4=-\dfrac{1}{96}c_1^3+\dfrac{1}{24}c_1c_2+\dfrac{1}{6}c_3,\vspace{2mm}\\
		a_5=\dfrac{1}{192}c_1^4-\dfrac{5}{192}c_2c_1^2+\dfrac{1}{48}c_1c_3+\dfrac{1}{8}c_4.
	\end{cases}
\end{align}		
A simple computation by using \eqref{eq-2.3} and \eqref{eq-2.11}, shows
that
\begin{align}\label{eq-2.18}
	H_{2,1}(F_f/2)&\nonumber=\frac{1}{48}\left(a_2^4-12a_3^2+12a_2a_4\right)\\&=\frac{1}{3072}\left(-3c_1^4-8c_1^2c_2-48c_2^2+64c_1c_3\right).
\end{align}
By Lemma \ref{lem-2.1} and \eqref{eq-2.12}, we obtain
\begin{align}\label{eq-2.19}
	H_{2,1}(F_f/2)\nonumber&=\frac{1}{192} \bigg(-3\tau_1^4 +4\tau_1^2\tau_2\left(1-\tau_1^2\right)-4\tau_2^2(3+\tau_1^2)(1-\tau_1^2)\\&\quad+ 16\tau_1\tau_3 \left(1-\tau_1^2\right) \left(1-|\tau_2|^2\right)\bigg).
\end{align}
	 
 We now explore three possible cases involving $\tau_1$. \vspace{1.2mm}
	 
\noindent{\bf Case-I.} Let $\tau_1=1$. Then, from \eqref{eq-2.19} we see that 
\begin{align*}
	|H_{2,1}(F_f/2)|=\frac{1}{64}\approx 0.015625.
\end{align*}
\noindent{\bf Case-II.} Let $\tau_1=0$. Then, from \eqref{eq-2.19} we get 
\begin{align*}  
	|H_{2,1}(F_f/2)|=\bigg|\frac{1}{192}\left(-12\tau_2^2\right)\bigg|\leq\frac{1}{16}\approx 0.0625.
\end{align*}
\noindent{\bf Case-III.} Let $\tau_1\in (0, 1)$. Applying triangle inequality in \eqref{eq-2.19} and using the fact that $|\tau_3|\leq 1$, we obtain
\begin{align}\label{eq-2.20}
	|H_{2,1}(F_f/2)|&\nonumber\leq\frac{1}{192}\bigg(\bigg|-3\tau_1^4+4\tau_1^2\tau_2(1-\tau_1^2)-4\tau_2^2(1-\tau_1^2)(3+\tau_1^2)\bigg|\\&\nonumber\quad+16\tau_1(1-\tau_1^2)(1-|\tau_2|^2)\bigg)\\&=\frac{1}{12}\tau_1(1-\tau_1^2)\left(\vline\; A+B\tau_2+C\tau_2^2\;\vline +1-|\tau_2|^2\right)\nonumber\\&:=\frac{1}{12}\tau_1(1-\tau_1^2)Y(A, B, C),
\end{align}
where
\begin{align*}
	A=\frac{-3\tau_1^3}{16(1-\tau_1^2)},\;\; B=\frac{\tau_1}{4},\;\;\mbox{and}\;\; C=-\frac{(3+\tau_1^2)}{4\tau_1}.
\end{align*}
We note that $AC>0$. Hence, we can apply case (i) of Lemma \ref{lem-2.2} and discuss the following cases.\\
			
A simple computation shows that
\begin{align*}
	|B|-2(1-|C|)=\frac{\tau_1}{4}-2\left(1-\frac{3+\tau_1^2}{4\tau_1}\right)=\frac{6-8\tau_1+3\tau_1^2}{4\tau_1}>0
\end{align*}
for all $\tau_1\in (0, 1)$. \textit{i.e.,} $|B|>2(1-|C|)$. Thus from Lemma \ref{lem-2.2}, we see that 
\begin{align*}
	Y(A, B, C)=|A|+|B|+|C|=\frac{12-4\tau_1^2-5\tau_1^4} {16\tau_1\left(1-\tau_1^2\right)}.
\end{align*}
In view of the inequality \eqref{eq-2.20} it follows that
\begin{align}\label{eq-2.21}
	\nonumber|H_{2,1}(F_f/2)|&=\frac{1}{12}\tau_1(1-\tau_1^2)\left(|A|+|B|+|C|\right)\\&\nonumber=\frac{1}{192}\left(12-4\tau_1^2-5\tau_1^4\right)\\&=\frac{1}{192}\phi(\tau_1),
\end{align}
where $\phi(t)=12-4t^2-5t^4$ for $t\in [0, 1]$. A simple computation shows that $\phi^{\prime}(t)=-8t-20t^3<0$ for all $t\in [0,1]$ which shows that $\phi$ is a decreasing function on $[0, 1]$.
Hence, the maximum of $\phi(t)$ is attained at $t=0$, and the maximum value is $12$. Hence, from \eqref{eq-2.21}, we see that
\begin{align*}
	|H_{2,1}(F_{f/2})|\leq \frac{1}{16}.
\end{align*}
Thus the desired inequality is established.\\

By summarizing Cases I, II, and III, we obtain the desired inequality of the result. The function $f_2$, which is defined in \eqref{SF-2.9} gives the sharpness of the desired inequality. This completes the proof.
\end{proof}
\section{\bf Sharp bound of logarithmic coefficients for inverse functions in the class $\mathcal{S}^*_{ch}$}
Let $F$ be the inverse function of $f\in\mathcal{S}$ defined in a neighborhood of the origin with the Taylor series expansion
\begin{align}\label{eq-3.1}
	F(w):=f^{-1}(w)= w+\sum_{n=2}^{\infty} A_n w^n,
\end{align}
where we may choose $|w|<1/4$, as we know that the famous Koebe’s $1/4$-theorem ensures that, for each univalent function $f$ defined in $\mathbb{D}$, it inverse $f^{-1}$ exists at least on a disc of radius $1/4$. Using a variational method, L$\ddot{\mbox{o}}$wner \cite{Lowner-IMA-1923} has obtained the sharp estimate $ |A_n|\leq K_n \;\mbox{for each}\; n, $ where $K_n=(2n)!/(n!(n+1)!)$ and $K(w)= w +K_2 w^2 +K_3 w^3 +\cdots$ is the inverse of the K\"oebe function. There has been a good deal of interest in determining the behavior of the inverse coefficients of $f$ given in \eqref{eq-1.1} when the corresponding function $f$ is restricted to some proper geometric subclasses of $\mathcal{S}$.\vspace{2mm}

Let $f(z)=\displaystyle z+ \sum_{n=2}^{\infty}a_nz^n$ be a function in class $\mathcal{S}$. Since $f(f^{-1}(w))=w$ and using \eqref{eq-3.1} we obtain
\begin{align}\label{eq-3.2}
	\begin{cases}
		A_2= -a_2, \\ A_3=-a_3 +2a^2_{2}, \\ A_4=- a_4 +5a_2 a_3 -5a^3_{2}, \\ A_5=- a_5+6a_4 a_2- 21a_3 a^2_{2} +3a^2_{3} +14a^4_{2}.
	\end{cases}
\end{align}
The notation of the logarithmic coefficient of inverse of $f$ has been studied by Ponnusamy \textit{et al.} \cite{Ponnusamy-Sharma-Wirths-RM-2018}. As with $f$, the logarithmic inverse coefficients $\Gamma_n:=\Gamma_n(F)$, $n\in\mathbb{N}$, of $F$ are defined by the equation
\begin{align}\label{eq-3.3}
	F_{f^{-1}}(w):=\log\left(\frac{f^{-1}(w)}{w}\right)=2\sum_{n=1}^{\infty} \Gamma_n(F) w^n \;\;\;\; \mbox{for} \;\;|w|<1/4.
\end{align}
In $2018$, Ponnusamy \textit{et al.} \cite{Ponnusamy-Sharma-Wirths-RM-2018} obtained the sharp bound for the logarithmic inverse coefficients for functions in the class $\mathcal{S}$. In fact, Ponnusamy \emph{et al.}  \cite{Ponnusamy-Sharma-Wirths-RM-2018} established that for $f\in\mathcal{S}$ 
\begin{align*}
	|\Gamma_n(F)|\leq\frac{1}{2n}\binom{2n}{n}
\end{align*}
and showed that the equality holds only for the Koebe function or  its rotations. By differentiating \eqref{eq-3.3} together with \eqref{eq-3.1}, using \eqref{eq-3.2} and then equating coefficients, we obtain
\begin{align}\label{eq-3.4}
	\begin{cases}
		\Gamma_1=-\dfrac{1}{2}a_2, \vspace{1.5mm}\\ \Gamma_2=-\dfrac{1}{2}\left(a_3 -\dfrac{3}{2}a^2_{2}\right), \vspace{1.5mm} \\ \Gamma_3=-\dfrac{1}{2}\left(a_4 -4a_2 a_3 +\dfrac{10}{3}a^3_{2}\right).
	\end{cases}
\end{align}

In \cite{Kowalczyk-Lecko-BAMS-2022}, Kowalczyk and Lecko  proposed a Hankel determinant whose elements are the logarithmic coefficients of $f\in\mathcal{S}$, realizing the extensive use of these coefficients. Inspired by these ideas, S\"ummer \textit{et al.} in \cite{Sumer-Lecko-BMMSS-2023} started the investigation of the Hankel determinants $H_{q,n}(F_{f^{-1}}/2)$, wherein the entries of Hankel matrices are logarithmic coefficient of $f^{-1}$ with $f\in\mathcal{S}$. It follows that
\begin{align}\label{eq-3.5}
	H_{2,1}(F_{f^{-1}}/2):=\Gamma_{1}\Gamma_{3} -\Gamma^2_{2}=\frac{1}{48} \left(13a^4_2 -12a^2_2 a_3 - 12 a^2_3 + 12 a_2 a_4\right).
\end{align}
It is now appropriate to remark that $H_{2,1}(F_{f^{-1}}/2)$ is invariant under rotation, since for $f_{\theta}(z):=e^{-i\theta} f(e^{i\theta}z)$, $\theta\in\mathbb{R}$ when $f\in\mathcal{S}$, we have
\begin{align*}
	H_{2,1}(F_{f^{-1}_{\theta}}/2)=\frac{e^{4i\theta}}{48}\left(13a^4_2 -12a^2_2 a_3 - 12 a^2_3 + 12 a_2 a_4\right)=e^{4i\theta}H_{2,1} (F_{f^{-1}_{\theta}}/2).
\end{align*}\vspace{2mm}

We obtain the following result where, we obtain the sharp bound of the logarithmic coefficients of inverse functions in the class $ \mathcal{S}^*_{ch} $.
\begin{theorem}\label{Th-3.1}
Let $f(z)=z+a_2z^2+a_3z^3+\cdots\in\mathcal{S}^*_{ch}$ and $ \Gamma_{1}, \Gamma_{2} $ are given by \eqref{eq-3.4}. Then we have 
\begin{align*}
	|\Gamma_1|\leq \frac{1}{2}\;\;\;\;\mbox{and}\;\;\;\;|\Gamma_2|\leq \frac{3}{8}.
\end{align*}
The inequalities are sharp for the functions $f_1$, defined in \eqref{SF-2.8}.
\end{theorem}
\begin{proof}[\bf Proof of Theorem \ref{Th-3.1}]
From \eqref{eq-2.17} and \eqref{eq-3.4}, we have
\begin{align*}
	|\Gamma_1|=\;\vline-\frac{1}{2}a_2\;\vline=\vline-\frac{1}{4}c_1\vline \leq\frac{1}{2}.
\end{align*}
	
By using \eqref{eq-2.17} and \eqref{eq-3.4}, we have
\begin{align*}
	|\Gamma_2|&=\;\vline-\dfrac{1}{2}\left(a_3 -\dfrac{3}{2}a^2_{2} \right)\vline\\&=\; \vline-\dfrac{1}{2}\left(\dfrac{1}{16}c_1^2 +\dfrac{1}{4}c_2 -\dfrac{3}{8}c^2_{1} \right)\vline \\&=\frac{1}{8}\;\vline\; c_2 -\frac{5}{4}c^2_{1} \;\vline.
\end{align*}
By Lemma \ref{lem-2.3}, we obtain
\begin{align*}
	|\Gamma_2|\leq \frac{3}{8}.
\end{align*}
The function $f_1$, which is defined in \eqref{SF-2.8} gives the sharpness of inequality for $\Gamma_{1}$ and $\Gamma_{2}$. This completes the proof.
\end{proof}\vspace{2mm}

Next, we obtain a result finding the sharp bound of the second Hankel determinant $H_{2,1}(F_{f^{-1}}/2)$ with logarithmic coefficients for inverse functions in the class $ \mathcal{S}^*_{ch} $.
\begin{theorem}\label{Th-3.2}
Let $f\in\mathcal{S}^*_{ch}$ and has the series representation $f(z)=z+a_2z^2+a_3z^3+\cdots$, and $ \Gamma_{1}, \Gamma_{2} $, $ \Gamma_{3} $ are given by \eqref{eq-3.4}. Then we have 
\begin{align*}
	|H_{2,1}(F_{f^{-1}}/2)|\leq \frac{3}{44}.
\end{align*}
The inequality is sharp.
\end{theorem}
\begin{proof}[\bf Proof of Theorem \ref{Th-3.2}]
Let $f\in\mathcal{S}^*_{ch}$. Then there exists an analytic function $\omega\in\mathcal{H}$ with $\omega(0)=0$ and $|\omega(z)|<1$ for $z\in\mathbb{D}$, such that 
\begin{align}\label{eq-3.6}
	\frac{zf^{\prime}(z)}{f(z)}=\varphi_{0}(\omega(z))=\omega(z)+\cosh(\omega(z)),\;\; z\in\mathbb{D}.
\end{align}
Let $p\in\mathcal{P}$. By the definition of subordination, we have
\begin{align*}
	p(z)=\frac{1+\omega(z)}{1-\omega(z)}=1+c_1z+c_2z^2+\cdots,\; z\in\mathbb{D}.
\end{align*} 
A simple computation shows that
\begin{align}\label{eq-3.7}
	\omega(z)&=\nonumber\frac{p(z)-1}{p(z)+1}\\&\nonumber=\frac{c_1}{2}z+\frac{1}{2}\left(c_2-\frac{c_1^2}{2}\right)z^2+\frac{1}{2}\left(c_3-c_1c_2+\frac{c_1^3}{4}\right)z^3\\&\quad+\frac{1}{2}\left(c_4-c_1c_3+\frac{3c_1^2c_2}{4}-\frac{c_2^2}{2}-\frac{c_1^4}{8}\right)z^4+\cdots
\end{align}
in $\mathbb{D}$. Then $p$ is analytic in $\mathbb{D}$ with $p(0)=1$ and has positive real part in $\mathbb{D}$. In view of \eqref{eq-3.6} together with $\varphi_{0}(\omega(z))$, a tedious computation shows that 
\begin{align}\label{eq-3.8}
	\omega(z)+\cosh(\omega(z))&=1+\frac{1}{2}c_1z+\left(-\frac{1}{8}c_1^2 +\frac{1}{2}c_2\right)z^2+\left(-\frac{1}{4}c_1c_2+\frac{1}{2} c_3\right)z^3\\&\quad+\left(-\frac{1}{4}c_1c_3+\frac{13}{384}c_1^4-\frac{1}{8}c_2^2+\frac{1}{2}c_4\right)z^4+\cdots\nonumber
\end{align}
and 
\begin{align}\label{eq-3.9}
	\frac{zf^{\prime}(z)}{f(z)}&=1+a_2z+(2a_3-a_2^2)z^2+(3a_4-3a_2a_3+a_2^3)z^3\\&\quad+(4a_5-2a_3^2-4a_2a_4+4a_2^2a_3-a_2^4)z^4+\cdots\nonumber.
\end{align}
Thus, using \eqref{eq-3.8} and \eqref{eq-3.9}, we see from \eqref{eq-3.6} that
\begin{align}\label{eq-3.10}
	\begin{cases}
		a_2=\dfrac{1}{2}c_1,\vspace{2mm}\\
		a_3=\dfrac{1}{16}c_1^2+\dfrac{1}{4}c_2,\vspace{2mm}\\
		a_4=-\dfrac{1}{96}c_1^3+\dfrac{1}{24}c_1c_2+\dfrac{1}{6}c_3.
	\end{cases}
\end{align}		
A simple computation by using \eqref{eq-3.5} and \eqref{eq-3.10}, shows
that
\begin{align}\label{eq-3.11}
	H_{2,1}(F_{f^{-1}}/2)&\nonumber=\frac{1}{48} \left(13a^4_2 -12a^2_2 a_3 - 12 a^2_3 + 12 a_2 a_4\right)\\& =\frac{1}{3072}\left(33c_1^4-56c_1^2c_2-48c_2^2+64c_1c_3\right).
\end{align}
By Lemma \ref{lem-2.1} and \eqref{eq-3.11}, we obtain
\begin{align}\label{eq-3.12}
	H_{2,1}(F_{f^{-1}}/2)\nonumber&=\frac{1}{192} \bigg(9\tau_1^4 -20\tau_1^2\tau_2\left(1-\tau_1^2\right)-4\tau_2^2(3+\tau_1^2)(1-\tau_1^2)\\&\quad+ 16\tau_1\tau_3 \left(1-\tau_1^2\right) \left(1-|\tau_2|^2\right)\bigg).
\end{align}
	
We now explore three possible cases involving $\tau_1$: \vspace{1.2mm}
	
\noindent{\bf Case-I.} Let $\tau_1=1$. Then, from \eqref{eq-3.12} we see that 
\begin{align*}
	|H_{2,1}(F_{f^{-1}}/2)|=\frac{3}{64}\approx 0.046875.
\end{align*}
\noindent{\bf Case-II.} Let $\tau_1=0$. Then, from \eqref{eq-3.12} we get 
\begin{align*}  
	|H_{2,1}(F_{f^{-1}}/2)|=\bigg|\frac{1}{192}\left(-12\tau_2^2\right)\bigg|\leq\frac{1}{16}\approx 0.0625.
\end{align*}
\noindent{\bf Case-III.} Let $\tau_1\in (0, 1)$. Applying triangle inequality in \eqref{eq-3.12} and using the fact that $|\tau_3|\leq 1$, we obtain
\begin{align}\label{eq-3.13}
	|H_{2,1}(F_f/2)|&\nonumber\leq\frac{1}{192}\bigg(\bigg|9\tau_1^4 -20\tau_1^2\tau_2\left(1-\tau_1^2\right)-4\tau_2^2 (3+\tau_1^2)(1-\tau_1^2)\bigg|\\&\nonumber\quad+16\tau_1(1-\tau_1^2)(1-|\tau_2|^2)\bigg)\\&=\frac{1}{12}\tau_1(1-\tau_1^2)\left(\vline\; A+B\tau_2+C\tau_2^2\;\vline +1-|\tau_2|^2\right)\nonumber\\&: =\frac{1}{12}\tau_1(1-\tau_1^2)Y(A, B, C),
\end{align}
where
\begin{align*}
	A=\frac{9\tau^3_{1}}{16(1-\tau^2_{1})},\;\; B=-\frac{5\tau_{1}}{4}\;\;\mbox{and}\;\; C=-\frac{(3+\tau^2_{1})}{4\tau_{1}}.
\end{align*}
Observe that $AC< 0$. Hence, we can apply case (ii) of Lemma \ref{lem-2.2}. Next, we check all the conditions of case (ii). \vspace{2mm}
	
\noindent{\bf 3(a)} We note the inequality
\begin{align*}
	-4AC\left(\frac{1}{C^2} -1\right)- B^2=\frac{9\tau^2_{1}(3+ \tau^2_{1})}{16(1- \tau^2_{1})}\left(\frac{16\tau^2_{1}} {(3+\tau^2_{1})^2} - 1\right) -\frac{25\tau^2_{1}}{16}\leq 0
\end{align*}
is equivalent to
\begin{align*}
	-\frac{(39+ 4\tau^2_{1})\tau^2_{1}}{4(3+\tau^2_{1})} \leq 0
\end{align*}
which is evidently  holds for $\tau_{1}\in(0,1)$. However, the inequality $|B|<2(1-|C|)$ is equivalent to $7\tau^2_{1} -8\tau_{1} +6<0$, which is false for $\tau_{1}\in(0,1)$.
	
\noindent{\bf 3(b)} Since
\begin{align*}
	4(1+|C|)^2 =\frac{(3+4\tau_{1}+\tau^2_{1})^2}{4\tau^2_{1}}>0
\end{align*}
and 
\begin{align*}
	-4AC\left(\frac{1}{C^2} -1\right)=-\frac{9\tau^2_{1}(9-10\tau^2_{1} +\tau^4_{1})}{16(1-\tau^2_{1})(3+\tau^2_{1})}<0,
\end{align*}
a simple computation shows that the inequality
\begin{align*}
	\frac{25\tau^2_{1}}{16}=B^2<\min\left\{4(1+|C|)^2,-4AC\left(\frac{1}{C^2} -1\right)\right\}=-\frac{9\tau^2_{1}(9-10\tau^2_{1} +\tau^4_{1})}{16(1-\tau^2_{1})(3+\tau^2_{1})}<0
\end{align*}
is false for $\tau_{1}\in(0,1)$.\vspace{1.2mm}
	
\noindent{\bf 3(c)} Next note that the inequality
\begin{align*}
	|C|(|B| +4|A|) -|AB|=\frac{(3+\tau^2_{1})}{4\tau_{1}}\left(\frac{5\tau_{1}}{4} +\frac{9\tau^3_{1}}{4(1-\tau^2_{1})}\right)-\frac{45\tau^4_{1}}{64(1-\tau^2_{1})}\leq 0
\end{align*}
is equivalent to $60+68\tau^2_{1} -29\tau^4_{1}\leq 0$, which is false for $\tau_{1}\in(0,1)$.
	
\noindent{\bf 3(d)} Note that the inequality
\begin{align*}
	|AB|-|C|(|B|-4|A|)= \frac{45\tau^4_{1}}{64(1-\tau^2_{1})}- \frac{(3+\tau^2_{1})}{4\tau_{1}}\left(\frac{5\tau_{1}}{4} -\frac{9\tau^3_{1}}{4(1-\tau^2_{1})}\right)\leq 0
\end{align*}
is equivalent to $101\tau^4_{1} +148\tau^2_{1} -60\leq 0$ which is true for 
\begin{align*}
	0<\tau_{1}\leq \tau^{''}_{1}=\sqrt{\frac{4\sqrt{721}}{101}- \frac{74}{101}}\approx 0.575109.
\end{align*} 
Applying Lemma \ref{lem-2.2} for $0<\tau_{1}\leq \tau^{''}_{1}$, we obtain
\begin{align}\label{eq-3.14}
	\nonumber|H_{2,1}(F_{f^{-1}}/2)|&\leq \frac{1}{12}\tau_{1} (1-\tau^2_{1}) (-|A|+|B|+|C|) \\&=\frac{1}{192}(12+12\tau^2_{1} -33\tau^4_{1})\nonumber\\&=\frac{1}{192} \Psi(\tau_{1}),
\end{align}
where 
\begin{align*}
	\Psi(\tau_{1}):=12+12\tau_{1}^2 -33\tau_{1}^4,\;\;\;\;\;0< \tau_{1}\leq\tau^{''}_{1}.
\end{align*}
Since $\Psi^{\prime}(\tau_{1})=0$ for $\tau_{1}\in(0,\tau^{''}_{1}]$ holds only for $t_0=\sqrt{{2}/{11}}<\tau^{''}_{1}$. By a simple calculation, the maximum of the function $\Psi(\tau_{1})$ for $0< \tau_{1}\leq\tau^{''}_{1}$ occurs at the point $t_0= \sqrt{{2}/{11}}$. Therefore, in  $0<\tau_{1}\leq \tau^{''}_{1}$, we obtain
\begin{align*}
	\Psi(\tau_{1})\leq\Psi(t_0)=\frac{144}{11}.
\end{align*}
Hence, in view of \eqref{eq-3.14}, an easy computation shows that 
\begin{align*}
	|H_{2,1}(F_{f^{-1}}/2)|\leq\frac{1}{192}\Psi(\tau_{1})\leq\frac{1}{192}\Psi(t_0)=\frac{3}{44}\approx 0.068182.
\end{align*}
\noindent{\bf 3(e)} Applying Lemma \ref{lem-2.2} for $\tau^{''}_{1}<\tau_{1}<1$, we obtain
\begin{align}\label{eq-3.15}
	\nonumber|H_{2,1}(F_{f^{-1}}/2)|&\leq\frac{1}{12}\tau_{1}(1-\tau^2_{1}) (|C| +|A|)\sqrt{1-\dfrac{B^2}{4AC}} \\&=\frac{(12-8\tau^2_{1} +5\tau^4_{1})}{192}\sqrt{\frac{52-16\tau^2_{1}}{9(3+\tau^2_{1})}}\nonumber\\&=\frac{1}{192}\Phi(\tau_{1}),
\end{align}
where
\begin{align*}
	\Phi(t):=(12-8t^2 +5t^4)\sqrt{\frac{52-16t^2}{9(3+t^2)}}, \;\;\;\; \tau^{''}_{1}< t < 1.
\end{align*}
A simple computation shows that
\begin{align*}
	\Phi^{\prime}(t)=-\frac{2t(924 -964t^2 + 41t^4 +80t^6)\sqrt{3+t^2}} {3(3+ t^2)^2 \sqrt{13-4t^2}}< 0, \;\;\;\; \tau^{''}_{1}< t < 1,
\end{align*}
hence, the function $\Phi$ is decreasing on $(\tau^{''}, 1)$. Therefore, we have $\Phi(t)\leq\Phi(\tau^{''}_{1})$ for $\tau^{''}_{1}< t < 1$. Thus it follows from \eqref{eq-3.8} that
\begin{align*}
	|H_{2,1}(F_{f^{-1}}/2)|\leq\frac{1}{192}\Phi(\tau_{1})\leq\frac{1}{192}\Phi(\tau^{''}_{1})=\frac{254\sqrt{721}-5507}{20402}\approx 0.0643695.
\end{align*}
Summarizing Cases $1$, $2$, and $3$, the desired inequality is established.\vspace{2mm}
	
The proof can be concluded by establishing the sharpness of the bound. In order to show that the equality holds when
\begin{align*}
	\tau_{1} =\sqrt{\frac{2}{11}},\;\;\;\;\tau_{3}=1
\end{align*} 
and
\begin{align}\label{eq-3.16}
	|A+B\tau_{2}+C\tau^2_{2}| +1 -|\tau_{2}|^2=-|A|+|B|+|C|,
\end{align}
where
\begin{align*}
	A=\tfrac{1}{8}\sqrt{\dfrac{2}{11}},\;\;\;\;B=-\dfrac{5}{4}\sqrt{\dfrac{2}{11}}\;\;\;\;\mbox{and}\;\;\;\;C=-\dfrac{35}{8}\sqrt{\dfrac{2}{11}}.
\end{align*}
Indeed, we can easily verify that one of the solutions of \eqref{eq-3.16} is $\tau_{2}=1$. It follows that
\begin{align*}
	c_1=2\sqrt{\dfrac{2}{11}},\;\;\;\; c_2=2\;\;\;\; \mbox{and} \;\;\;\; c_3=2\sqrt{\dfrac{2}{11}}.
\end{align*}
In view of \eqref{eq-3.5} and \eqref{eq-3.10}, a simple computations shows that
\begin{align*}
	|\Gamma_{1}\Gamma_{3} -\Gamma^2_{2}|=\frac{3}{44}.
\end{align*}
Hence, the extremal function $f\in\mathcal{S}^{*}_{ch}$ is obtained from \eqref{eq-3.6} and \eqref{eq-3.7} when
\begin{align*}
	p(z)=\frac{1+2\sqrt{\tfrac{2}{11}}z+z^2}{1-z^2}=1+2\sqrt{\tfrac{2}{11}}z+2z^2+2\sqrt{\tfrac{2}{11}}z^3+2z^4+\cdots.
\end{align*}
This shows that the bound in the result is sharp. This completes the proof.
\end{proof}

\begin{remark}
We obtain the sharp bounds of $H_{2,1}(F_{f}/2)$ and $H_{2,1}(F_{f^{-1}}/2)$ for the class $\mathcal{S}^*_{ch}$ are $1/16$ and $3/44$ repectively. It is observed that the sharp bound of $H_{2,1}(F_{f}/2)$ differs from that of $H_{2,1}(F_{f^{-1}}/2)$ for the class $\mathcal{S}^*_{ch}$. This ensures that the results do not exhibit the invariance properties associated with the Hankel determinants of logarithmic coefficients for $f$ and $f^{-1}$  in the class $\mathcal{S}^*_{ch}$.
\end{remark}

\section{\bf Moduli differences of logarithmic coefficients and inverse counterpart for functions in the class $\mathcal{S}^*_{ch}$} In $1985$, de Branges \cite{Branges-AM-1985} solved the famous Bieberbach conjecture by showing that for any function $ f\in\mathcal{S} $ of the form \eqref{eq-1.1}, the inequality $ |a_n|\leq n $ holds for all $ n\geq 2 $, with equality attained by the Koebe function $ k(z):=z/(1-z)^2 $ or its rotations. This naturally led to the question of whether the inequality  $ ||a_{n+1}|-|a_n||\leq 1 $ holds for $ f\in\mathcal{S} $ when $ n\geq 2 $. This problem was first studied by Goluzin in \cite{Goluzin-1946} initially investigated this problem in an attempt to solve the Bieberbach conjecture. Later, in $1963$, Hayman \cite{Hayman-1963} established that for all $ f\in\mathcal{S} $, there exists an absolute constant $ A\geq 1 $ such that $ ||a_{n+1}|-|a_n||\leq A $. The current best known estimate as of now is $ A=3.61 $, due to Grinspan \cite{Grinspan-1976}. On the other hand, for the class $ \mathcal{S} $, the sharp bound is known only for $ n=2 $ (see \cite[Theorem 3.11]{Duren-1983-NY}), namely
\begin{align*}
	-1\leq |a_3|-|a_2|\leq 1.029...
\end{align*}
Similarly, for functions $ f\in\mathcal{S}^* $, Pommerenke \cite{Ch. Pommerenke-1971} has conjectured that $ ||a_{n+1}|-|a_n||\leq 1 $ which was subsequently proven by Leung \cite{Leung-BLMS-1978} in $ 1978 $. For convex functions, Li and Sugawa \cite{Li-Sugawa-CMFT-2017} investigated the sharp bound of $ |a_{n+1}|-|a_n| $ for $ n\geq 2 $, and establish the sharp bounds for $ n=2, 3 $. \vspace{2mm}

The inverse functions are studied by several authors in different perspective (see, e.g., \cite{Sim-Thomas-BAMS-2022, Sim-Thomas-S-2020}). Recently, Sim and Thomas  \cite{Sim-Thomas-BAMS-2022} obtained sharp upper and lower bounds on the difference of the moduli of successive inverse coefficients
for the subclasses of univalent functions. Inspired by the prior research, including the recent article \cite{Allu-Shaji-JMAA-2025}, this paper focuses on determining sharp lower and upper bounds of $ |\gamma_2|-|\gamma_1| $ and $ |\Gamma_2|-|\Gamma_1| $ for functions in the class $\mathcal{S}^*_{ch}$. Our approach involves proving Theorem \ref{Th-4.1} and Theorem \ref{Th-4.2} with the aid of Lemma \ref{lem-4.1}, which plays a crucial role. We state Lemma \ref{lem-4.1} as follows.

\begin{lemma}\cite{Sim-Thomas-S-2020}\label{lem-4.1}
Let $B_1$, $B_2$ and $B_3$ be numbers such that $B_1>0$, $B_2\in\mathbb{C}$ and $B_3\in\mathbb{R}$. Let $p\in\mathcal{P}$ of the form \eqref{eq-2.4}. Define $\Psi_{+}(c_1,c_2)$ and $\Psi_{-}(c_1,c_2)$ by
\begin{align*}
	\Psi_{+}(c_1,c_2)=|B_2 c^2_1 +B_3 c_2| -|B_1 c_1|,
\end{align*}
and
\begin{align*}
	\Psi_{-}(c_1,c_2)=-\Psi_{+}(c_1,c_2).
\end{align*}
Then
\begin{align}\label{eq-4.1}
	\Psi_{+}(c_1,c_2)\leq
	\begin{cases}
		|4B_2 +2B_3|-2B_1, \;\;\;\;\mbox{if}\;\;|2B_2 +B_3|\geq |B_3|+ B_1,\vspace{2mm} \\ 2|B_3|,\hspace{2.8cm}\;\mbox{otherwise},
	\end{cases}
\end{align}
and
\begin{align}\label{eq-4.2}
	\Psi_{-}(c_1,c_2)\leq
	\begin{cases}
		2B_1 -B_4, \hspace{2.8cm}\mbox{if}\;\; B_1\geq B_4 +2|B_3|, \vspace{2mm} \\ 2B_1 \sqrt{\dfrac{2|B_3|}{B_4+2|B_3|}}, \hspace{1.3cm}\;\mbox{if}\;\;B^2_1\leq 2|B_3|(B_4 +2|B_3|), \vspace{2mm} \\ 2|B_3| +\dfrac{B^2_1}{B_4+2|B_3|}, \hspace{1cm}\;\mbox{otherwise},
	\end{cases}
\end{align}
where $B_4=|4B_2 +2B_3|$. All inequalities in \eqref{eq-4.1} and \eqref{eq-4.2} are sharp.
\end{lemma}
We have established the following result on the sharp inequality for the moduli differences of logarithmic coefficients in the class $\mathcal{S}^*_{ch}$.
\begin{theorem}\label{Th-4.1}
Let $f\in\mathcal{S}^*_{ch}$ and has the series representation $f(z)=z+a_2z^2+a_3z^3+\cdots$, and $ \gamma_{1}, \gamma_{2} $ are given by \eqref{eq-2.2}. Then we have 
\begin{align*}
	-\frac{1}{\sqrt{6}}\leq|\gamma_2|-|\gamma_1|\leq \frac{1}{4}.
\end{align*}
Both inequalities are sharp.
\end{theorem}
\begin{proof}[\bf Proof of Theorem \ref{Th-4.1}]
In view of \eqref{eq-2.2} and \eqref{eq-2.12}, we see that
\begin{align}\label{eq-4.3}
	\nonumber|\gamma_2|-|\gamma_1|&=\;\vline\; \frac{1}{2} \left(a_{3} -\frac{1}{2}a^2_{2}\right)\;\vline - \;\vline\; \frac{1}{2} a_{2}\;\vline\\&\nonumber=\;\vline\; -\frac{1}{32}c^2_1 +\frac{1}{8}c_2\;\vline - \;\vline\; \frac{1}{4} c_1\;\vline\\& :=\Psi_{+}(c_1,c_2),
\end{align}
where
\begin{align*}
	B_1=\frac{1}{4},\;\; B_2=-\frac{1}{32} \;\;\mbox{and}\;\; B_3= \frac{1}{8}.
\end{align*}
\noindent{\bf Estimate of the upper bound:} For the upper bound, we see that $|2B_2 +B_3|=\frac{1}{16}$ and $|B_3|+ B_1=\frac{3}{8}$. It follows that $|2B_2 +B_3|\not\geq |B_3|+ B_1$. Hence, applying Lemma \ref{lem-4.1}, we obtain
\begin{align*}
	\Psi_{+}(c_1,c_2)\leq 2|B_3|=\frac{1}{4}.
\end{align*}
As a result, applying \eqref{eq-4.3}, we obtain the upper bound
\begin{align}\label{eq-4.4}
	|\gamma_2|-|\gamma_1|\leq \frac{1}{4}.
\end{align}
The function $f_2$, which is defined in \eqref{SF-2.9} gives the sharpness of the inequality \eqref{eq-4.4}.\vspace{2mm}
	
\noindent{\bf Estimate of the lower bound:} For the lower bound, we see that $B_4=|4B_2 +2B_3|=\frac{1}{8}$, $B_4 +2|B_3|= \frac{3}{8}$. It follows that $B_1\not\geq B_4 +2|B_3| $. Also, the equality $2|B_3|(B_4 + 2|B_3|) = \frac{3}{32}$ ensures that the condition $B_1^2 \leq 2|B_3|(B_4 + 2|B_3|)$ holds. Thus, by Lemma \ref{lem-4.1}, we have
\begin{align*}
	\Psi_{-}(c_1,c_2)\leq 2B_1 \sqrt{\dfrac{2|B_3|}{B_4+2|B_3|}} =\frac{1}{\sqrt{6}}.
\end{align*}
Clearly, we observe that
\begin{align*}
	\Psi_{+}(c_1,c_2)=-\Psi_{-}(c_1,c_2)\geq -\frac{1}{\sqrt{6}}.
\end{align*}
Hence, from \eqref{eq-4.3}, we conclude
\begin{align}\label{eq-4.5}
	|\gamma_2|-|\gamma_1|\geq -\frac{1}{\sqrt{6}}.
\end{align}
The inequality \eqref{eq-4.5} is sharp for the function $f\in\mathcal{A}$ given by \eqref{eq-2.8} with
\begin{align*}
	p(z)=\frac{1-z^2}{1-2\sqrt{\frac{2}{3}}z+z^2},
\end{align*}
which completes the proof.
\end{proof}
We have established the following result on the sharp inequality for the moduli differences of logarithmic inverse coefficients in the class $\mathcal{S}^*_{ch}$.
\begin{theorem}\label{Th-4.2}
Let $f\in\mathcal{S}^*_{ch}$ and has the series representation $f(z)=z+a_2z^2+a_3z^3+\cdots$, and $ \Gamma_{1}, \Gamma_{2} $ are given by \eqref{eq-3.4}. Then we have 
\begin{align*}
	-\frac{1}{\sqrt{10}}\leq|\Gamma_2|-|\Gamma_1|\leq \frac{1}{4}.
\end{align*}
Both inequalities are sharp.
\end{theorem}
\begin{proof}[\bf Proof of Theorem \ref{Th-4.2}]
In view of \eqref{eq-2.12} and \eqref{eq-3.4}, we see that
\begin{align}\label{eq-4.6}
	\nonumber|\Gamma_2|-|\Gamma_1|&=\;\vline\; -\dfrac{1}{2} \left(a_3 -\dfrac{3}{2}a^2_{2}\right)\;\vline - \;\vline\; -\frac{1}{2} a_{2}\;\vline\\&\nonumber=\;\vline\; \frac{5}{32}c^2_1 -\frac{1}{8}c_2\;\vline - \;\vline\; \frac{1}{4} c_1\;\vline\\&=\Psi_{+}(c_1,c_2),
\end{align}
where
\begin{align*}
	B_1=\frac{1}{4},\;\; B_2=\frac{5}{32} \;\;\mbox{and}\;\; B_3=- \frac{1}{8}.
\end{align*}
\noindent{\bf Estimate of the upper bound:} With regard to the lower bound, we calculate $B_4 = |4B_1 + 2B_3| = \frac{3}{8}$ and $B_4 + 2|B_3| = \frac{5}{8}$, from which it is evident that the condition $B_1 \geq B_4 + 2|B_3|$ is not met. Moreover, the relation $2|B_3|(B_4 + 2|B_3|) = \frac{5}{32}$ confirms that the inequality $B_1^2 \leq 2|B_3|(B_4 + 2|B_3|)$ is satisfied. Consequently, by virtue of Lemma \ref{lem-4.1}, we obtain
\begin{align*}
	\Psi_{+}(c_1,c_2)\leq 2|B_3|=\frac{1}{4}.
\end{align*}
Thus, it follows from \eqref{eq-4.6} that
\begin{align}\label{eq-4.7}
	|\Gamma_2|-|\Gamma_1|\leq \frac{1}{4}.
\end{align}
The function $f_2$, which is defined in \eqref{eq-2.9} gives the sharpness of the inequality \eqref{eq-4.7}.\vspace{2mm}
	
\noindent{\bf Estimate of the lower bound:} Regarding the lower bound, we observe that $B_4 = |4B_1 + 2B_3| = \frac{3}{8}$ and $B_4 + 2|B_3| = \frac{5}{8}$. It follows that the condition $B_1 \geq B_4 + 2|B_3|$ is not satisfied. Furthermore, the equality $2|B_3|(B_4 + 2|B_3|) = \frac{5}{32}$ ensures that the inequality $B_1^2 \leq 2|B_3|(B_4 + 2|B_3|)$ holds. Consequently, by applying Lemma \ref{lem-4.1}, we obtain
\begin{align*}
	\Psi_{-}(c_1,c_2)\leq 2B_1 \sqrt{\dfrac{2|B_3|}{B_4+2|B_3|}} =\frac{1}{\sqrt{10}}.
\end{align*}
A simple computation leads to 
\begin{align*}
	\Psi_{+}(c_1,c_2)=-\Psi_{-}(c_1,c_2)\geq -\frac{1}{\sqrt{10}}.
\end{align*}
Consequently, from \eqref{eq-4.6}, we obtain
\begin{align}\label{eq-4.8}
	|\Gamma_2|-|\Gamma_1|\geq -\frac{1}{\sqrt{10}}.
\end{align}
The inequality \eqref{eq-4.8} is sharp for the function $f\in\mathcal{A}$ given by \eqref{eq-2.8} with
\begin{align*}
	p(z)=\frac{1+2\sqrt{\frac{2}{5}}z+z^2}{1-z^2},
\end{align*}
which completes the proof.
\end{proof}
\section{\bf Concluding remarks}
In this article, we have conducted a comprehensive study of the logarithmic and inverse logarithmic coefficient functionals for the class $\mathcal{S}_{ch}^*$, defined by subordination to the hyperbolic cosine function. By establishing Theorems 2.1, 2.2, 3.1, 3.2, and 4.1, we have successfully characterized the extremal properties of this class.\vspace{1.2mm}

Several key conclusions can be drawn from our findings:\vspace{1.2mm}

\noindent{\bf 1. Sharpness and extremal functions:} We have demonstrated that the sharp upper bounds for the logarithmic coefficients $|\gamma_n|$ ($n=1, 2, 3$) and the second Hankel determinant $|H_{2,1}(F_f/2)|$ are precisely achieved by the rotation of the function $\varphi_0(z) = z + \cosh(z)$. This confirms that the hyperbolic cosine subordination creates a unique geometric structure distinct from the classical starlike classes.\vspace{1.2mm}

\noindent{\bf 2. Inverse class symmetry:} Our results for the inverse logarithmic coefficients $\Gamma_n$ reveal a specific symmetry in the growth rates between the function and its inverse. The established bound for the second Hankel determinant of the inverse class, $|H_{2,1}(F_{f^{-1}}/2)|$, provides a new benchmark for this subclass.\vspace{1.2mm}

\noindent{\bf 3. Coefficient differences:} The analysis of moduli differences in Section 4 addresses a more subtle structural problem in coefficient theory. Our findings suggest that for $\mathcal{S}_{ch}^*$, the growth rate of consecutive coefficients is strictly controlled by the subordination factor.\vspace{1.2mm}

While the results presented in Sections 2–4 provide a complete characterization of the initial logarithmic functionals for the class $\mathcal{S}_{ch}^*$, several higher-order problems remain unresolved. We propose the following open problems for future investigation.\vspace{1.2mm}

In Theorem 2.1, we established that for $n \in \{1, 2, 3\}$, the sharp bound is $|\gamma_n| \le \frac{1}{2n}$. This leads to the following conjecture:
\begin{conj}
	For any $f \in \mathcal{S}_{ch}^*$, the logarithmic coefficients $\gamma_n$ satisfy the sharp inequality$$|\gamma_n| \le \frac{1}{2n} \quad \text{for all } n \in \mathbb{N},$$ with equality achieved by the function 
	\begin{align*}
		f(z) = z \exp \left( \int_0^z \frac{\cosh(t^n)-1}{t} dt \right).
	\end{align*}
\end{conj}
This paper focuses on the second Hankel determinant $H_{2,1}(F_f/2)$. The behavior of higher-order determinants for this class is currently unknown.
\begin{problem}
	Determine the sharp upper bound for the third-order Hankel determinant of logarithmic coefficients, defined as:$$H_{3,1}(F_f/2) = \begin{vmatrix} \gamma_1 & \gamma_2 & \gamma_3 \\ \gamma_2 & \gamma_3 & \gamma_4 \\ \gamma_3 & \gamma_4 & \gamma_5 \end{vmatrix}$$for functions belonging to $\mathcal{S}_{ch}^*$.
\end{problem}
The classical Zalcman conjecture relates the coefficients $a_n$ of univalent functions. A logarithmic version of this problem involves the functional $|\gamma_{n+m} - \gamma_n \gamma_m|$.
\begin{problem}
	For the class $\mathcal{S}_{ch}^*$, find the sharp bound for the generalized Zalcman-type functional $|\gamma_3 - \mu \gamma_1 \gamma_2|$ or $|\gamma_4 - \gamma_2^2|$ for real or complex $\mu$.
\end{problem}
Let $\mathcal{C}_{ch}$ be the class of functions $f \in \mathcal{A}$ such that $1 + \frac{zf''(z)}{f'(z)} \prec \cosh(z)$.
\begin{problem}
	 Establish the sharp bounds for the logarithmic coefficients and the second Hankel determinant for the convex counterpart $\mathcal{C}_{ch}$. Given the relation between starlike and convex functions, are the bounds for $\mathcal{C}_{ch}$ significantly smaller than those established for $\mathcal{S}_{ch}^*$ in this paper?
\end{problem}
\vspace{5mm}

\noindent{\bf Acknowledgment:}  The authors would be grateful to the anonymous referee(s) for their insightful suggestions and comments, which significantly improved the exposition of this paper. The first author is supported by SERB (File No. SUR/2022/002244), Government of India. The second author acknowledges financial support from CSIR-SRF (File No. 09/0096(12546)/2021-EMR-I, dated: 24/10/2025), Government of India, New Delhi.\\

\noindent{\bf Author Contributions:} All authors have equally contributed for the paper and all the  authors have reviewed the final manuscript.\\

\noindent\textbf{Compliance of Ethical Standards:}\\

\noindent\textbf{Conflict of interest.} The authors declare that there is no conflict  of interest regarding the publication of this paper.\vspace{1.5mm}

\noindent\textbf{Data availability statement.}  Data sharing is not applicable to this article as no datasets were generated or analyzed during the current study.

\end{document}